\documentclass[12pt]{article}
\usepackage{amsmath}
\usepackage{amsfonts}
\usepackage{amsthm}
\usepackage{graphicx}
\usepackage{overpic}
\usepackage{a4wide}
\usepackage{amssymb} 
\usepackage{pictex}
\usepackage{rotating}  
\usepackage{color}
\usepackage{etex} %Pour Pb ``no room for a new dim''
\usepackage{enumitem}
\usepackage{accents}

\definecolor{refkey}{rgb}{0,0,1}
\definecolor{labelkey}{rgb}{1,0,0}

\usepackage{comment}

\numberwithin{equation}{section}

\newtheorem{theorem}{Theorem}[section]
\newtheorem{proposition}[theorem]{Proposition}
\newtheorem{lemma}[theorem]{Lemma}
\newtheorem{corollary}[theorem]{Corollary}
\newtheorem{Definition}[theorem]{Definition}
\newenvironment{definition}{\begin{Definition}\rm}{\end{Definition}}
\newtheorem{Remark}[theorem]{Remark}
\newenvironment{remark}{\begin{Remark}\rm}{\end{Remark}}
\newtheorem{Example}[theorem]{Example}

\newtheorem{RHproblem}[theorem]{RH problem}

\usepackage{color}

\newcommand{\C}{\mathbb{C}}

\newcommand{\R}{\mathbb{R}}

\newcommand{\MM}{\mathcal M}
\newcommand{\PP}{\mathcal P}

\renewcommand{\tilde}{\widetilde}

\usepackage[bookmarksopen, naturalnames]{hyperref}

\begin{document}
\title{A large deviation principle for weighted Riesz interactions}
\author{Tom Bloom, Norman Levenberg and Franck Wielonsky}
\maketitle 

\begin{abstract}
We prove a large deviation principle for the sequence of push-forwards of empirical measures in the setting of Riesz potential interactions on compact subsets $K$ in $\R^d$ with continuous external fields. Our results are valid for base measures on $K$ satisfying a strong Bernstein-Markov type property for Riesz potentials. Furthermore, we give sufficient conditions on $K$  (which are satisfied if $K$ is a smooth submanifold) so that a measure on $K$ which satisfies a mass-density condition will also satisfy this strong Bernstein-Markov property.
\end{abstract}

\section{Introduction} Fix a positive integer $d>2$. Let $W(y)=\frac{1}{|y|^{\alpha}}$ where $y=(y_1,...,y_d)\in \R^d$, $|y|^2=\sum_{j=1}^dy_j^2$, and $0<\alpha <d$. For $K\subset \R^d$ a compact set of positive Riesz $\alpha-$capacity and $Q:K\to \R$ continuous, we consider the ensemble of probability measures $Prob_n$ on $K^n$:
\begin{equation}\label{probnint} Prob_n:=\frac{1}{Z_n} \exp {\bigl[-\sum_{1\leq i\neq j\leq n}W(x_i-x_j)\bigr]} \cdot \exp{\bigl[-2n\sum_{j=1}^nQ(x_j)\bigr]}d\nu(x_1)\cdots d\nu(x_n)\end{equation}
where $d\nu$ is a measure on $K$ and $Z_n$ is a normalizing constant. Our main result, stated at the end of the introduction, is a large deviation principle for the sequence $\{\sigma_n=(j_n)_*(Prob_n)\}$ of probability measures on $\mathcal M(K)$, the space of probability measures on $K$, where $j_n:  K^{n} \to \mathcal M(K)$ is the empirical measure map $j_n(x_1,...,x_{n})=\frac{1}{n}\sum_{j=1}^{n} \delta_{x_j}$. Weighted Riesz interactions have been studied by many authors, e.g., \cite{CGZ} and \cite{LS}, but generally the situation considered is $K=\R^d$ and $d\nu$ is Lebesgue measure (and $Q(x)$ satisfies a growth condition as $|x|\to \infty$). 

In this paper, we follow the technique utilized, e.g., in \cite{VELD} and \cite{LPLD}. We first discuss weighted Riesz potential-theoretic notions in the next section such as the weighted Riesz energy functional $I^Q$. This entails a weighted Riesz energy minimization problem $\inf_{\mu\in \mathcal M(K)} I^Q(\mu)$ with minimizer $\mu_{K,Q}$ and a corresponding discretization. Section 3 forms the heart of the paper; there we prove a Bernstein-type estimate (Proposition \ref{key}) on ``polynomial-like'' functions arising from our discretization process. This leads to a sufficient mass-density condition on a measure $\mu$ on certain compact sets $K\subset \R^d$ so that we have a good comparability between supremum norms and $L^1(\mu)$ norms of weighted versions of such functions (Theorem \ref{bmq} on strong Bernstein-Markov measures for Riesz potentials on $K$; see Remark \ref{strongBM} for this definition). With these ingredients in hand, the consequences, such as one-point correlation asymptotics (Corollary \ref{417}) and a large deviation principle, follow: 

\begin{theorem} Let $\nu$ be a strong Bernstein-Markov measure for $K$ and $Q$ continuous on $K$. The sequence $\{\sigma_n=(j_n)_*(Prob_n)\}$ of probability measures on $\mathcal M(K)$ satisfies a large deviation principle with speed $n^{2}$ and good rate function $\mathcal I:=\mathcal I_{K,Q}$ where, for $\mu \in \mathcal M(K)$,
\begin{equation*}%\label{ratefcnlform}
\mathcal I(\mu)=I^Q(\mu)-I^Q(\mu_{K,Q}).
\end{equation*}
%see (\ref{w=j=i}).
 \end{theorem}

\section{Riesz potential theory}

Let $K\subset \R^d$ be compact and let $\mathcal M(K)$ be the set of probability measures on $K$ endowed with the topology of weak convergence. Fix $0<\alpha <d$. We consider the Riesz energy minimization problem:
$$\inf_{\mu\in \mathcal M(K)}I(\mu)$$
where
$$I(\mu):=\int_K \int_K  \frac{1}{|x-y|^{\alpha}}d\mu(x) d\mu(y)$$ 
is the Riesz energy of $\mu$. We will restrict to $0<\alpha <d$ for the rest of the paper. If there exists $\mu \in \mathcal M(K)$ with $I(\mu)<\infty$ we say $K$ has {\it positive Riesz $\alpha-$capacity} (and henceforth we drop the ``$\alpha$''). We remark that if one considers the cone ${\mathcal M}_+$ of all positive measures on $\R^d$ (not necessarily with compact support), it is known (cf., \cite[Chapter I]{La}, \cite{CGZ}) that 
\begin{enumerate}
\item for $\mu \in {\mathcal M}_+, \ I(\mu)\geq 0$;
\item $I(\mu)=0$ if and only if $\mu=0$;
\item $\mu \to I(\mu)$ is strictly convex on $\{\mu\in {\mathcal M}_+: I(\mu)<\infty\}$.
\end{enumerate}

\noindent Note that Landkof \cite{La} replaces our $\alpha$ with $d-\alpha$.

\begin{remark}\label{future} We add for future use that 2. holds for signed measures $\mu=\mu_1-\mu_2$ with $\mu(K)=0$ when the mixed energy $I(\mu_1,\mu_2):=\int_K \int_K  \frac{1}{|x-y|^{\alpha}}d\mu_1(x) d\mu_2(y)$ is finite; this is in Theorem 1.15, p. 79 of \cite{La}.\end{remark} 

We also consider a weighted Riesz energy minimization problem. Given a compact set $K$ of positive Riesz capacity, and a lower semicontinuous function $Q$ on $K$ with $\{x\in K:Q(x)<\infty\}$ of positive Riesz capacity (we write $Q\in \mathcal A(K)$), we consider
$$\inf_{\mu\in \mathcal M(K)} I^Q(\mu)$$
where
$$I^Q(\mu):=I(\mu)+2\int_KQ(x)d\mu(x):=\int_K \int_K  \frac{1}{|x-y|^{\alpha}}d\mu(x) d\mu(y)+2\int_KQ(x)d\mu(x).$$
In later portions of this paper, we will restrict to $Q\in C(K)$ (continuous functions on $K$).

\begin{remark} In \cite{CGZ} the authors consider the situation where $K=\R^d$ and $Q(x)$ satisfies a growth condition as $|x|\to \infty$. Their Theorem 1.2 gives general results, in this setting, for the weighted energy minimization problems, while their Theorem 1.1 is a large deviation principle using  $Prob_n$ measures as in (\ref{probnint}) which are taken with respect to Lebesgue measure on $\R^d$. We will allow general (possibly singular) measures $\nu$ in (\ref{probnint}) for our large deviation principle, Theorem \ref{ldp}. See also \cite{LS} for further results. 

\end{remark}

We define the Riesz potential associated to a positive measure $\mu$ on $K$:
$$U^{\mu}(x):=\int_K  \frac{1}{|x-y|^{\alpha}}d\mu(y).$$ 
The following properties hold:
\begin{enumerate}
\item The Riesz potential is a lower semicontinuous function on $\R^d$. It is superharmonic in $\R^{d}$ for $0<\alpha\leq d-2$ and subharmonic outside $K$ for $d-2<\alpha$, (\cite[Theorem I.1.4 p. 66]{La}). 
\item For $d-2< \alpha <d$, we have a {\it domination principle} (\cite{La}, Theorem 1.2.9, p. 115): for $\mu$ a measure whose potential $U^{\mu}$ is finite $\mu-$a.e., and $u$ a superharmonic function, if the inequality $U^{\mu}\leq u$ holds $\mu-$a.e., then it holds everywhere. 
\item Also, for $d-2\leq \alpha <d$, we have a {\it maximum principle} (\cite{La}, Theorem 1.10, p. 71): for $\mu$ a measure with $U^{\mu}\leq M$ $\mu-$a.e., this estimate holds everywhere. 
\item There is a {\it weak maximum principle} (\cite{La}, Theorem 1.5, p. 66): for all $0<\alpha <d$, given $\mu$ a measure with $U^{\mu}\leq M$ on supp$(\mu)$, we have $U^{\mu}\leq 2^{\alpha}M$ on $\R^d$. 
\item This last property is sufficient to prove a {\it continuity property} of Riesz potentials (\cite{La}, Theorem 1.7, p. 69): for all $0<\alpha <d$, given $\mu$ a measure with $U^{\mu}$ continuous on supp$(\mu)$, we have $U^{\mu}$ is continuous on $\R^d$.   
\end{enumerate}
Throughout, unless otherwise specified, we assume $0<\alpha<d$.
Following the arguments on pp. 27-33 in \cite{ST} for weighted logarithmic potential theory (or for $K=\R^d$ in \cite{CGZ}) we have the following.

\begin{theorem}\label{frost}
For $K\subset \R^d$ compact and of positive Riesz capacity, and for $Q\in \mathcal A(K)$, 
\begin{enumerate}[noitemsep,nolistsep]
\item $V_w:=\inf_{\mu\in \mathcal M(K)}I^Q(\mu)$ is finite;
\item there exists a unique weighted equilibrium measure $\mu_{K,Q} \in \mathcal M(K)$ with $I^Q(\mu_{K,Q})=V_w$;
\item the support $S_w:=$supp$(\mu_{K,Q})$ is contained in $\{x\in K: Q(x) <\infty\}$ and $S_w$ is of positive Riesz capacity;
\item if we let $F_w:=V_w -\int_K Q(x)d\mu_{K,Q}(x)$, then
\begin{align*}
U^{\mu_{K,Q}}(x)+Q(x) & \geq F_w\text{ on }K\setminus P\text{ where }P
\text{ is of zero Riesz capacity (possibly empty);}\\
U^{\mu_{K,Q}}(x)+Q(x) & \leq F_w\text{ for all }x\in S_w.
\end{align*}
\end{enumerate}
\end{theorem}

\begin{remark} In the proof of the Frostman-type property 4. in \cite{ST}, one simply replaces ``q.e.'' -- off of a set of positive logarithmic capacity in $\C$ -- by ``off of a set of zero Riesz capacity'' as the essential property used is the existence of a measure of finite logarithmic energy on a compact subset of a set of positive logarithmic capacity in $\C$. 
\end{remark}

\begin{remark} \label{needed} For $K\subset \R^d$ compact and of positive Riesz capacity, if $\mu\in \mathcal M(K)$ with $I(\mu)<\infty$, one can consider a weighted energy minimization problem with the {\it upper semicontinuous} weight $Q=-U^{\mu}$. Following the proof of Lemma 5.1 of \cite{LPLD}, the minimum is attained (uniquely) by the measure $\mu$; i.e., 
$$I(\mu)+2\int_KQ(x)d\mu(x)\leq I(\nu)+2\int_KQ(x)d\nu(x) \ \hbox{for all} \ \nu\in \mathcal M(K)$$
with equality if and only if $\nu=\mu$. This uses Remark \ref{future}.

\end{remark}

The ``converse'' to 4. of Theorem \ref{frost} holds as well. This is stated/proved in \cite{CGZ} in their setting (Theorem 1.2 (1.10) and (1.11)). 
\begin{proposition}\label{frostchar} Let $K\subset \R^d$ be compact and of positive Riesz capacity and let $Q\in \mathcal A(K)$. For a measure $\mu\in \mathcal M(K)$, if there exists a constant $C$ such that
$$U^{\mu}(x)+Q(x)  \geq C\text{ on }K\setminus P \ \text{ where }P
\text{ is of zero Riesz capacity}$$ and 
$$U^{\mu}(x)+Q(x)  \leq C \text{ for all }x\in \text{supp}\mu,$$ 
then $\mu =\mu_{K,Q}$. 
\end{proposition}
\begin{proof} We write 
$$\mu_{K,Q}=\mu+(\mu_{K,Q}-\mu).$$ 
Then
$$I^Q(\mu)\geq I^Q(\mu_{K,Q}) = I^Q(\mu) +I(\mu_{K,Q}-\mu)+2R$$
with
\begin{align*}
R := & \int_K\bigl[\int_K \frac{1}{|x-y|^{\alpha}}d\mu(y) +Q(x)\bigr]d(\mu_{K,Q}-\mu)(x) 
\\
 = & \int_K(U^\mu(x) +Q(x))d(\mu_{K,Q}-\mu)(x).
%-\int_K(p_\mu(x)  +p_{f_*(\mu)}(f(x))+Q(x))d\mu(x)
\end{align*}
Note that the above computation is justified. Indeed, from the assumptions $I^{Q}(\mu)<\infty$ and $\mu$ has compact support, the quantities $I^{Q}(\mu)$, $I(\mu)$, $\int Qd\mu$, and the mixed energy $I(\mu,\mu_{K,Q})$ are all finite. Making use of
the inequalities in the hypotheses, we conclude that
$$R\geq  C\int_K d\mu_{K,Q}-C\int_K d\mu=0.$$
Recall that $I(\mu_{K,Q}-\mu)\geq 0$ with equality if and only if $\mu_{K,Q}=\mu$ (Remark \ref{future}). Thus 
$$I^Q(\mu) \geq I^Q(\mu_{K,Q}) 
%= E^Q(\mu) +I(\mu_{K,Q}-\mu)+ I(f_*(\mu_{K,Q}-\mu))+2R
\geq I^Q(\mu)$$
so that equality holds throughout, and $I^Q(\mu) = I^Q(\mu_{K,Q})$, from which
follows $\mu =\mu_{K,Q}$. 
\end{proof}

\begin{remark} For $d-2\leq \alpha <d$, since we have a maximum principle, using 4. of Theorem \ref{frost} and following the argument in the proof of Theorem 4 in \cite{BDS}, for any $\mu \in \mathcal M(K)$ we have
$$\text{``}\inf_{x\in S_w}\text{''}[U^{\mu}(x)+Q(x)]\leq F_w$$
and
$$\sup_{x\in \text{supp}\mu}[U^{\mu}(x)+Q(x)]\geq F_w.$$
Here, $\text{``}\inf_{x\in S}\text{''} F(x)$ denotes the largest number $L$ such that on $S$ the real-valued function $F$ takes values smaller than $L$ only on a set of zero Riesz capacity. The corresponding version of  
this result for certain weights on all of $\R^d$ is stated as equations (1.14) and (1.15) in \cite{CGZ}.

\end{remark}

As in \cite{ST}, we can characterize the compact sets $K\subset \R^d$ which arise as supports of a weighted energy minimizing measure.

\begin{theorem} Let $K\subset \R^d$ be compact and of positive Riesz capacity at each point of $K$. Then there exists $Q\in \mathcal A(K)$ such that $S_w=K$. 
\end{theorem}

\begin{proof} Suppose we can find a probability measure $\mu$ with support $K$ such that $U^{\mu}$ is continuous on $K$ (and hence continuous on $\R^d$ by the aforementioned continuity property of Riesz potentials in \cite{La}, Theorem 1.7, p. 69). Then taking $Q(x):=-U^{\mu}(x)$ on $K$, we have 
$$U^{\mu}(x)+Q(x)=0 \ \hbox{on} \ K$$
so that by Proposition \ref{frostchar} we have $\mu =\mu_{K,Q}$. 

To construct such a $\mu$, we follow the arguments in Lemma I.6.10 and Corollary I.6.11 in \cite{ST} (we do not need the final statement in Lemma I.6.10). In particular, for any compact $S\subset \R^d$ of positive Riesz capacity, we obtain a finite, positive measure $\nu$ with support in $S$ such that $U^{\nu}$ is continuous. Using this, we follow exactly the proof of Theorem IV.1.1 in \cite{ST}.

\end{proof}

Next we discretize: for $n\geq 2$, let 
$$VDM_n^Q(x_1,...,x_n):= \exp {\bigl[-\sum_{1\leq i \neq j\leq n}W(x_i-x_j)\bigr]} \cdot \exp{\bigl[-2n\sum_{j=1}^nQ(x_j)\bigr]}$$
$$=\exp {\bigl[-\sum_{1\leq i \neq j\leq n}\frac{1}{|x_i-x_j|^{\alpha}}\bigr]} \cdot \exp{\bigl[-2n\sum_{j=1}^nQ(x_j)\bigr]}=:\exp[-L_n(x_1,...,x_n)]$$
$$\hbox{where} \ L_n(x_1,...,x_n)= \sum_{1\leq i\neq j\leq n}\frac{1}{|x_i-x_j|^{\alpha}} +2n\sum_{j=1}^nQ(x_j).$$
Thus
$$\frac{1}{n(n-1)} L_n(x_1,...,x_n)= \frac{1}{n(n-1)}\sum_{1\leq i\neq j\leq n}\frac{1}{|x_i-x_j|^{\alpha}} +\frac{2}{n-1} \sum_{j=1}^nQ(x_j)$$ 
is the approximate weighted Riesz energy of $\mu_n :=\tfrac{1}{n}\sum_{j=1}^n \delta_{x_j}$; i.e., where we ignore ``diagonal'' terms (which make the true Riesz energy of $\mu_n$ infinite).

We define the $n$-th weighted diameter $\delta_{n}^{Q}(K)$ by
\begin{equation}\label{deltan}\delta_n^Q(K):=\sup_{x_1,...,x_n\in K} 
VDM_n^Q(x_1,...,x_n)^{1/n^2}.\end{equation}
We will show the limit of these quantities exists, and this {\it weighted transfinite diameter of $K$ with respect to $Q$} satisfies
$$\delta^Q(K):=\lim_{n\to \infty} \delta_n^Q(K) = \exp {(-V_w)}=\exp{(-I^Q(\mu_{K,Q}))}.$$
By upper semicontinuity of $(x_1,...,x_n)\to -L_n(x_1,...,x_n)$ on $K^n$ and $-Q$ on $K$ the supremum in (\ref{deltan}) is attained; we call any collection of $n$ points of $K$ at which the maximum is
attained {\it weighted Fekete points} of order $n$ for $K,Q$. Following the proofs of Propositions 3.1--3.3 of \cite[Section 3]{VELD} we have:

\begin{theorem} \label{sec3} Given $K\subset \R^d$ compact and of positive Riesz capacity and $Q\in 
\mathcal A(K)$,
\begin{enumerate}
\item if $\{\mu_n=\frac{1}{n}\sum_{j=1}^n\delta_{x_j^{(n)}}\}\subset \mathcal M(K)$ converge weakly to $\mu\in \mathcal M(K)$, then 
\begin{equation}\label{upboundVDM}
\limsup_{n\to \infty} VDM_n^Q(x_1^{(n)},...,x_n^{(n)})^{1/n^2}\leq \exp{(-I^Q(\mu))};
\end{equation}
\item we have $\delta^Q(K):=\lim_{n\to \infty} \delta_n^Q(K)$ exists and 
$$\delta^Q(K)=\exp{(-V_w)}=\exp{(-I^Q(\mu_{K,Q}))};$$
\item if $\{x_j^{(n)}\}_{j=1,...,n; \ n=2,3,...}\subset K$ and 
$$\lim_{n\to \infty} VDM_n^Q(x_1^{(n)},...,x_n^{(n)})^{1/n^2}= \exp{(-V_w)}$$
then 
$$\mu_n=\frac{1}{n}\sum_{j=1}^n\delta_{x_j^{(n)}}\to \mu_{K,Q} \ \hbox{weakly}.$$

\end{enumerate}

\end{theorem}  

\begin{proof} To prove 1., note that $\mu_n\times \mu_n \to \mu \times \mu$ in $\mathcal M(K)\times \mathcal M(K)$. Exploiting  lower semicontinuity of the kernel 
$$(x,y)\to W(x-y)=\frac{1}{|x-y|^{\alpha}},$$
for $M\in \R$ we define the continuous kernel 
$$h_M(x,y):= \min [M,W(x-y)]$$
and we have
$$I^Q(\mu)=\lim_{M\to \infty} \int_K \int_K h_M(x,y) d\mu(x)d\mu(y) + 2\int_K Qd\mu$$
$$\leq \lim_{M\to \infty} \liminf_{n\to \infty} \bigl( \int_K \int_K h_M(x,y) d\mu_n(x)d\mu_n(y) + 2\int_K Qd\mu_n\bigr)$$
$$\leq \lim_{M\to \infty} \liminf_{n\to \infty} \bigl[\frac{1}{n^2}  \bigl( nM + \sum_{j\not = k}\frac{1}{|x_j^{(n)}-x_k^{(n)}|^{\alpha}}\bigr)+ \frac{2}{n}\sum_{j=1}^nQ(x_j^{(n)}) \bigr].$$
Here we used lower semicontinuity of $Q$ to conclude $ \int_K Qd\mu\leq \liminf_{n\to \infty} \int_K Qd\mu_n$. Thus
$$ \limsup_{n\to \infty} VDM_n^Q(x_1^{(n)},...,x_n^{(n)})^{1/n^2} \leq e^{-I^Q(\mu)}.$$

To prove 2., let
$$D_n(K):=\inf_{x_1,...,x_n\in K}\frac{1}{n(n-1)} L_n(x_1,...,x_n).$$
Note that 
\begin{equation}\label{dnvdm}
D_n(K)\leq \frac{-1}{n(n-1)}\log VDM_n^Q(x_1,...,x_n), \ \hbox{all} \ (x_1,...,x_n)\in K^n
\end{equation}
with equality for weighted Fekete points of order $n$ for $K,Q$. Fix $n$ points $x_1,...,x_n\in K$. Then
$$D_n(K) \leq \frac{1}{n(n-1)} L_n(x_1,...,x_n)=\frac{1}{n(n-1)}\sum_{1\leq i \neq j\leq n}\frac{1}{|x_i-x_j|^{\alpha}} +\frac{2}{n-1} \sum_{j=1}^nQ(x_j).$$
For $\mu \in \mathcal M(K)$, integrate both sides with respect to $\prod_{i<j}d\mu(x_i)d\mu(x_j)$:
$$D_n(K) \leq I(\mu)+ \frac{2n}{n-1}\int_K Qd\mu.$$
Thus
\begin{equation}\label{upest} \limsup_{n\to \infty} D_n(K)\leq I^Q(\mu).\end{equation}
On the other hand, taking $\mu_n := \frac{1}{n}\sum_{j=1}^n \delta_{x_j^{(n)}}$ where 
$D_n(K)=\frac{1}{n(n-1)} L_n(x_1^{(n)},...,x_n^{(n)})$ (i.e., weighted Fekete points of order $n$ for $K,Q$), if $\mu$ is any weak limit of this sequence  then 1. implies that 
$$\limsup_{n\to \infty} VDM_n^Q(x_1^{(n)},...,x_n^{(n)})^{1/n^2}\leq \exp{(-I^Q(\mu))}.$$
From (\ref{upest}) and the equality portion of (\ref{dnvdm}), 
\begin{equation}\label{upperbound} \limsup_{n\to \infty} \frac{-1}{n(n-1)}\log VDM_n^Q(x_1^{(n)},...,x_n^{(n)})=\limsup_{n\to \infty} [-D_n(K)]\leq I^Q(\mu).\end{equation}
Thus $\lim_{n\to \infty} VDM_n^Q(x_1^{(n)},...,x_n^{(n)})^{1/n^2}$ exists and equals $\exp{(-I^Q(\mu))}$ where $\mu$ is any weak limit of weighted Fekete measures. Since $\mu_{K,Q}$ is the unique weighted energy minimizing measure we claim that 2. follows: for if $\sigma\in \mathcal M(K)$ is arbitrary, applying 1. to $\sigma$ and using (\ref{upperbound}) shows
$$\exp{(-I^Q(\sigma))}\leq  \liminf_{n\to \infty} VDM_n^Q(x_1^{(n)},...,x_n^{(n)})^{1/n^2}$$
$$\leq
\limsup_{n\to \infty} VDM_n^Q(x_1^{(n)},...,x_n^{(n)})^{1/n^2}\leq \exp{(-I^Q(\mu))}.$$

Item 3. follows from 1. and 2. 
\end{proof}
\begin{remark} Note that item 1. required lower semicontinuity of the kernel 
$$(x,y)\to W(x-y)=\frac{1}{|x-y|^{\alpha}}$$
and of $Q$ while item 2. (and hence 3.) required the existence and uniqueness of the weighted energy minimizing measure $\mu_{K,Q}$.
\end{remark}

%\begin{remark} If $Q\equiv 0$, we write
%$$VDM_n(x_1,...,x_n):= \exp {\bigl[-\sum_{1\leq i<j\leq n}W(x_i-x_j)\bigr]}$$
%\end{remark}

\section{Bernstein-type estimate} 
In this section, we always assume $Q\in C(K)$. 

\vskip6pt

If we fix $n-1$ points $x_2,...,x_n\in K$ and consider
$$y\to  VDM_n^Q(y,x_2,...,x_n), \ y=(y_1,...,y_d)\in \R^d,$$
then this function is of the form 
$$g(x_2,...,x_n)\cdot f_n^Q(y):=g(x_2,...,x_n)\cdot \exp  {\bigl[-\sum_{j=2}^n\frac{1}{|y-x_j|^{\alpha}}-2nQ(y)\bigr]}$$
where 
$$g(x_2,...,x_n)=\exp{\bigl(-2\sum_{j=2}^n Q(x_j)}\bigr)\cdot \exp  {\bigl(-\sum_{2\leq j \neq k\leq n }\frac{1}{|x_k-x_j|^{\alpha}}\bigr)}.$$ 
For notation, fixing $K$ and $Q$,  we let 
\begin{equation}\label{pnq} \mathcal P_n^Q:=\left\{f_n^Q(y):=
\exp\left(-\sum_{j=2}^n\frac{1}{|y-x_j|^{\alpha}}-2nQ(y)\right): x_2,...,x_n\in K\right\}.\end{equation}
If $Q\equiv 0$, we simply write $\mathcal P_n$ and $f_n$.

%we assume one of the following hypotheses $(H)$ on $K\subset \R^d$:
We recall the definition of the box-counting (or Minkowski) dimension of a bounded subset $K$ 
of $\R^{d}$, see e.g. \cite[Chapter 3]{Fal}. Let $N_{\delta}(K)$ be the smallest number of closed balls of radius $\delta$ which can cover $K$. The lower and upper box-counting dimensions of $K$ are defined as
$$\underline{\dim}_{B}K=\liminf_{\delta\to 0}\frac{\log N_{\delta}(K)}{-\log\delta},\qquad
\overline{\dim}_{B}K=\limsup_{\delta\to 0}\frac{\log N_{\delta}(K)}{-\log\delta}.
$$ 
If the limits are equal, the common value $\dim_{B}$ is refered to as the box-counting dimension of $K$. In particular, a smooth, compact $m$-dimensional submanifold of $\R^{d}$ (or a subdomain of it) has $\dim_{B}$ equal to $m$. 
\begin{Remark}
For a general bounded set $K$ one has
\begin{equation}\label{compare-dim}
\dim_{H}K\leq\underline{\dim}_{B}K\leq\overline{\dim}_{B}K,
\end{equation}
where $\dim_{H}$ denotes the Hausdorff dimension. Equalities hold for many regular sets, in particular for Ahlfors regular sets, i.e. sets which support a Borel regular measure $\mu$ such that, for some constant $C>1$, and all $x\in K$,
$$C^{-1}r^{s}\leq\mu(B(x,r))\leq Cr^{s},$$
where the exponent $s$ is the common dimension. 
\end{Remark}
%\begin{enumerate}
%\item $K$ is a smooth, compact $m-$dimensional submanifold of $\R^d$ with $1\leq m \leq d-1$;
%\item $K=\bar D$ where $D$ is a bounded, smoothly bounded subdomain of a smooth $m-$dimensional submanifold of $\R^d$ with $1\leq m \leq d$.
%\end{enumerate}
\begin{lemma}\label{minor}
Assume $K\subset\R^{d}$ is a compact set of positive lower box-counting dimension $m=\underline\dim_{B}K>0$. For every $n\geq 2$, there exists a constant $A_{n}>0$ such that 
$$\forall f_{n}\in\PP_{n},\qquad\|f_{n}\|_{K}\geq A_{n},$$
and for $n$ large enough, one may take $A_{n}=\exp(-2^{\alpha}n^{1+2\alpha/m})$.
\end{lemma}
\begin{proof}
Denote by $M_{n}(K)$ the smallest radius $\delta$ such that $K$ can be covered by $n$ closed balls of radius $\delta$. Since $n\geq N_{M_{n}(K)}(K)$, one has, for $n$ large enough,
$$
\frac{\log n}{-\log M_{n}(K)}\geq
\frac{\log N_{M_{n}(K)}(K)}{-\log M_{n}(K)}\geq \frac{m}{2},
$$
or equivalently
$$
M_{n}(K)\geq \left(\frac1n\right)^{2/m}.
$$
Since $n$ balls of radius $M_{n}(K)/2$ cannot cover $K$, we deduce that for $f_n(y):=
e^{\left(-\sum_{i=2}^n\frac{1}{|y-x_i|^{\alpha}}\right)}$, 
$$\exists y\in K,\quad,\forall i=2,\ldots,n,\quad|y-x_{i}|>\frac{M_{n}(K)}{2},$$
which implies
$$\forall f_{n}\in\PP_{n},\quad\|f_{n}\|_{K}\geq \exp\left(-\frac{2^{\alpha}n}{M_{n}(K)^{\alpha}}\right)
\geq\exp(-2^{\alpha}n^{\frac{2\alpha}{m}+1}).
$$
Since, by assumption, for every $n$, $M_{n}(K)$ is positive, the first inequality shows the existence of the constant $A_{n}$, which may be chosen as the first exponential. For $n$ large enough, it may also be chosen as the second exponential, expressed in terms of the lower box-counting dimension $m$ of $K$.
\end{proof}
We next obtain a uniform Bernstein-type estimate for $f_n \in \mathcal P_n$. 
%%%%%%%%%%%%%%%%%%%%%%%%%%%%%%%
\begin{proposition}\label{key}
Assume $K\subset\R^{d}$ is a compact set of positive lower box-counting dimension $m>0$. Then,
%the following inequality holds true:
$$\forall n\geq2,\quad\exists C_{n}>0,\quad\forall f_{n}\in\PP_{n},\qquad\|\nabla f_{n}\|_{2,K}:=||[\sum_{j=1}^n |\partial f_{n}/\partial y_{j}|^2]^{1/2}||_K \leq 
C_{n}\|f_{n}\|_{K},$$
where, for $n$ large enough, $C_{n}=C_{\alpha}n^{\beta}$ with a constant $C_{\alpha}$ depending on $\alpha$ only, and $\beta=2+1/\alpha+2\alpha/m+2/m$. In particular, $C_{n}^{1/n}\to 1$ as $n\to\infty$.
\end{proposition}
Note that by (\ref{compare-dim}) the above proposition applies as soon as $K$ has positive Hausdorff dimension.
\begin{proof}
We estimate $\partial f_{n}/\partial y_{1}$ on $\R^{d}$. We have
$$\frac{\partial f_{n}}{\partial y_{1}}(y)=\alpha f_{n}(y)\sum_{j=2}^{n}\frac
{y_{1}-(x_{j})_{1}}{|y-x_{j}|^{\alpha+2}},$$
so that
$$\left|\frac{\partial f_{n}}{\partial y_{1}}(y)\right|\leq\alpha f_{n}(y)\sum_{j=2}^{n}\frac{1}{|y-x_{j}|^{\alpha+1}}.$$
%Assume first $\|f_{n}\|_{K}\geq 1/2$. 
Denote by $M_{\alpha}$ the maximum of the function
$$\rho>0\to\frac{1}{\rho^{\alpha+1}}\exp\Big(-\frac{1}{\rho^{\alpha}}\Big).$$
Using the fact that $\exp\Big(-\frac{1}{|y-x_{j}|^{\alpha}}\Big)\leq 1$ for $j=2,...,n$ so that 
$$\exp\Big(-\sum_{j=2}^n\frac{1}{|y-x_{j}|^{\alpha}}\Big)\leq \exp\Big(-\frac{1}{|y-x_{j}|^{\alpha}}\Big),$$
we obtain the estimate
$$\left|\frac{\partial f_{n}}{\partial y_{1}}(y)\right|\leq\alpha 
\sum_{j=2}^{n}\exp\Big(-\frac{1}{|y-x_{j}|^{\alpha}}\Big)\frac{1}{|y-x_{j}|^{\alpha+1}}
\leq (n-1)\alpha M_{\alpha}.$$
Thus
\begin{equation}\label{abs-bound}
\forall y\in\R^{d},\qquad\|\nabla f_{n}(y)\|_{2}=[\sum_{j=1}^n |\frac{\partial f_{n}}{\partial y_{j}}(y)|^2]^{1/2}\leq n\sqrt{d}\alpha M_{\alpha}.
\end{equation}
%In particular, if $\|f_{n}\|_{K}\geq 1$, we deduce
%$$\|\nabla f_{n}\|_{2,K}\leq n\sqrt{d}\alpha M_{\alpha}\|f_{n}\|_{K},$$
%which proves the statement of the lemma for $f_{n}\in\PP_{n}(K)$ such that $\|f_{n}\|_{K}\geq 1$.
%Now, we assume $\|f_{n}\|_{K} < 1$. 
For any $\lambda>0$, we consider the functions
$$f_{n}(y)=\exp\Big(-\sum_{j=2}^{n}\frac{1}{|y-x_{j}|^{\alpha}}\Big),\qquad
F_{n}(y)=\exp\Big(-\sum_{j=2}^{n}\frac{1}{|y-\lambda x_{j}|^{\alpha}}\Big),$$
on $K$ and $\lambda K$. It is easily checked that, for $y\in K$,
$$F_{n}(\lambda y)=\left(f_{n}\left(y\right)\right)^{\lambda^{-\alpha}},\qquad
%\lambda\frac{\partial F_{n}}{\partial y_{1}}(\lambda y)
%=\lambda^{-\alpha}\left(f_{n}\left(y\right)\right)^{\lambda^{-\alpha}-1}\frac{\partial f_{n}}{\partial y_{1}}\left(y\right),
\lambda\nabla F_{n}(\lambda y)=\lambda^{-\alpha}\left(f_{n}\left(y\right)\right)^{\lambda^{-\alpha}-1}\nabla f_{n}\left(y\right).
$$
Thus,
$$\|\nabla f_{n}\|_{2,K}\leq\lambda^{\alpha+1}
\|\nabla F_{n}\|_{2,\lambda K}\|f_{n}\|_{K}^{1-\lambda^{-\alpha}}.$$
%We choose $\mu$ such that $\|f_{n}\|_{K}\leq\mu<1$ and $\lambda<1$ such that
%$\|f_{n}\|_{K}^{\lambda^{-\alpha}}=\mu$.
We choose $\lambda$ so as to minimize $\lambda^{\alpha+1}\|f_{n}\|_{K}^{-\lambda^{-\alpha}}$.
One may check that the function
$$\lambda>0\to\lambda^{\alpha+1}\|f_{n}\|_{K}^{-\lambda^{-\alpha}}$$
has a unique minimum which is
$$\left(\frac{e\alpha\log(1/\|f_{n}\|_{K})}{\alpha+1}\right)^{1+1/\alpha}.$$
Hence,
$$\|\nabla f_{n}\|_{2,K}\leq C_{\alpha}n\log(1/\|f_{n}\|_{K})^{1+1/\alpha}\|f_{n}\|_{K},
$$
where we have used (\ref{abs-bound}) applied to the function $F_{n}$.
Making use of Lemma \ref{minor} gives the result.
% for $n$ large enough. On the other hand, for a fixed value of $n$, the existence of a constant $C_{n}$, uniform on $\PP_{n}(K)$, follows from the above inequality and the fact that the minimum of $\|f_{n}\|_{K}$ over functions $f_{n}\in\PP_{n}(K)$ is positive (by a reasoning as in the proof of Lemma \ref{minor} using the fact that $M_{n}(K)>0$).
\end{proof}
%%%%%%%%%%%%%%%%%%%%%%%%%%%%%%%%
For the next result, we need to apply the Bernstein estimate locally. Thus, we introduce a local box-counting dimension of a subset $A$ of $\R^{d}$,
$$\forall x\in A,\quad\dim_{B}(A,x):=\lim_{r\to 0}\dim_{B}(A\cap B(x,r)).$$

In the next theorem, we assume that the compact set $K$ satisfies the following hypotheses:
\begin{enumerate}
\item There exists a $\rho>0$ such that, for all $x\in K$, $\dim_{B}(K,x)\geq\rho$.
\item %we can also find 
For $\delta=\delta(K)$ sufficiently small one can find $L=L(K,\delta)>0$ so that for any $x\in K$ and $y\in K\cap B(x,\delta)$, there is a rectifiable curve $\gamma\subset K $ joining $x$ to $y$ of length at most $L|x-y|$ where $L$ is independent of $x,y$. This property of a set is often called local quasiconvexity in the literature, see e.g. \cite{Hei}. It implies in particular that $K$ is locally path connected.
\end{enumerate}

We can now prove:

\begin{theorem} \label{bmq} Let $\mu$ be a positive measure on $K$ of finite total mass and suppose $\mu$ satisfies the following {\it mass density condition}: there exist constants $T,c, r_0 >0$ such that for all $x\in K$, 
\begin{equation}\label{mass2} \frac{\mu(B(x,r))}{r^T}\geq c >0 \ \hbox{for all} \ r\leq r_0.\end{equation}
Then for any $Q\in C(K)$, 
\begin{equation}\label{BMQ} ||f_n^Q||_K \leq M_n \int_K f_n^Q(x)d\mu(x) \ \hbox{for all} \ f_n^Q \in \mathcal P_n^Q \end{equation}
where $M_n=M_n(Q)$ satisfies $M_n^{1/n}\to 1$.
\end{theorem}

\begin{proof} %We take $\delta>0$ as above (and so that 1. holds). 
Fix $f_n^Q=f_ne^{-2nQ} \in \mathcal P_n^Q$ and let $w\in K$ be a point with
$$||f_ne^{-2nQ}||_K = f_n(w)e^{-2nQ(w)}.$$
Given $\epsilon >0$ sufficiently small, there is a $\delta>0$ such that
$$|Q(a)-Q(b)| \leq \epsilon \ \hbox{if} \ a,b\in K \ \hbox{with} \ |a-b|<\delta.$$
We have
$$f_n(w)e^{-2nQ(w)} \geq  f_n(x)e^{-2nQ(x)} \ \hbox{for all} \ x\in K.$$
Thus if $|x-w|<\delta$ we have 
$$f_n(w)e^{2n\epsilon} \geq f_n(x)$$
and hence
\begin{equation}\label{easy} ||f_n||_{B(w,\delta)}\leq f_n(w)e^{2n\epsilon}.\end{equation}

For $x\in B(w,\delta)$, 
$$f_n(x)- f_n(w) = \int_0^1 \nabla f_n (r(t))r'(t)dt$$
where $t\to r(t)$ is a smooth curve joining $w$ to $x$ as above. Applying Lemma \ref{key} in $B(w,\delta)$, 
$$|f_n(x)- f_n(w)| \leq || \nabla f_n ||_{2,B(w,\delta)} \int_0^1 \|r'(t)\|_{2}dt$$
$$\leq || \nabla f_n ||_{2,B(w,\delta)}L|x-w|\leq C_n L|x-w|\cdot ||f_n||_{B(w,\delta)}.$$
Now for $n$ large so that $e^{-3n\epsilon}< \min(\delta,r_0)$ and $C_nL\cdot e^{-n\epsilon}<1/2$, for $x\in B(w,e^{-3n\epsilon})\subset B(w,\delta)$ we may apply this estimate together with (\ref{easy}) to conclude that
$$|f_n(x)- f_n(w)| \leq C_nL\cdot e^{-3n\epsilon}\cdot ||f_n||_{B(w,\delta)}\leq C_nL\cdot e^{-3n\epsilon}\cdot f_n(w)e^{2n\epsilon}\leq \frac{1}{2}f_n(w).$$
Hence
$$f_n(x)\geq \frac{1}{2} f_n(w) \ \hbox{for large} \ n \ \hbox{if} \ x\in B(w,e^{-3n\epsilon}).$$
Finally,
 $$%||f_ne^{-2nQ}||_{L^1(\mu)}=
 \int_K f_n(x)e^{-2nQ(x)} d\mu(x)\geq \int_{B(w,e^{-3n\epsilon})}f_n(x)e^{-2nQ(x)} d\mu(x)$$
 $$\geq \frac{1}{2}f_n(w)e^{-2nQ(w)}e^{-2n\epsilon}\mu(B(w,e^{-3n\epsilon}))$$
 $$\geq \frac{1}{2}f_n(w)e^{-2nQ(w)}e^{-2n\epsilon}ce^{-3nT\epsilon}\geq ||f_ne^{-2nQ}||_K\cdot \frac{c}{2}e^{-n\epsilon(2+3T)}.$$
\end{proof}

\begin{remark} \label{strongBM} We call a measure $\mu$ satisfying (\ref{BMQ}) for each $Q\in C(K)$ a {\it strong Bernstein-Markov measure} (for Riesz potentials) on $K$. As an example, for $K$ a smooth, compact $m-$dimensional submanifold of $\R^d$, the Hausdorff $m-$measure (or equivalently its volume form) is a strong Bernstein-Markov measure on $K$. A special case of the results in this section was proved in \cite{DRNA}.
\end{remark}

%{\red It is unclear how much we may weaken the hypothesis that $K$ be a smooth, compact $m-$dimensional submanifold of $\R^d$.}

\section{Free energy asymptotics and a.s. convergence}

Let $K\subset \R^d$ be compact and of positive Riesz capacity; $Q\in C(K)$; and fix a measure $\mu$ on $K$ satisfying (\ref{BMQ}). For each $n=2,3,...$, define
$$Z_n=Z_n(K,Q,\mu):= \int_{K^n} VDM_n^Q(x_1,x_2,...,x_n)d\mu(x_1) \cdots d\mu(x_n).$$

\begin{proposition}\label{znasymp} With $K,Q$ and $\mu$ as above,
\begin{equation}\label{zneqn} 
\lim_{n\to \infty} Z_n^{1/n^2}= \delta^Q(K)=\exp {(-V_w)}.
\end{equation}

\end{proposition}

\begin{proof} %We apply Theorem \ref{bmq}. 
We make use of inequality (\ref{BMQ}).
Fix a set of $n$ points $a_1,...,a_n\in K$ with 
$$\max_{x_1,...,x_n\in K}VDM_n^Q(x_1,x_2,...,x_n)=VDM_n^Q(a_1,a_2,...,a_n).$$
The function
$$y\to VDM_n^Q(y,a_2,...,a_n)$$
is, up to a multiplicative constant, a function in $\mathcal P_n^Q$ which attains its maximum value on $K$ at $y=a_1$. %Applying Theorem \ref{bmq}, 
Using (\ref{BMQ}),
$$VDM_n^Q(a_1,a_2,...,a_n)\leq M_n \int_K VDM_n^Q(y,a_2,...,a_n)d\mu(y).$$
Now consider, for each fixed $y\in K$, 
$$z\to VDM_n^Q(y,z,a_3,...,a_n).$$
This is, up to a multiplicative constant, a function in $\mathcal P_n^Q$; and
$$VDM_n^Q(y,a_2,a_3,...,a_n)\leq \max_{z\in K}VDM_n^Q(y,z,a_3,...,a_n).$$
%Applying Theorem \ref{bmq}, 
Using (\ref{BMQ}) again,
$$VDM_n^Q(y,a_2,a_3,...,a_n)\leq M_n \int_K VDM_n^Q(y,z,a_3,...,a_n)d\mu(z)$$
which gives
$$VDM_n^Q(a_1,a_2,...,a_n)\leq M_n^2\int_K \int_K VDM_n^Q(y,z,a_3,...,a_n)d\mu(z)d\mu(y).$$
Repeating this argument $n-2$ times gives
$$VDM_n^Q(a_1,a_2,...,a_n)\leq M_n^n Z_n.$$
On the other hand,
$$Z_n \leq VDM_n^Q(a_1,a_2,...,a_n) \mu(K)^n.$$
Combining these last two displayed inequalities with 2. of Theorem \ref{sec3} and the fact that $M_n^{1/n}\to 1$ gives the result.

\end{proof}

We define a probability measure $Prob_n$ on $K^{n}$ as follows: for a Borel set $A\subset K^{n}$,
\begin{equation}\label{probn}Prob_n(A):=\frac{1}{Z_n}\cdot \int_A  |VDM_n^Q(x_1,...,x_n)|^2  d\mu(x_1)\cdots d\mu(x_n).
\end{equation}
This coincides with (\ref{probnint}). From Proposition \ref{znasymp} we obtain the following estimate.

\begin{corollary} \label{largedev} With $K,Q$ and $\mu$ as above, given $\eta >0$, define
 \begin{equation}\label{aketa}
 A_{n,\eta}:=\{(x_1,...,x_n)\in K^{n}: VDM_n^Q(x_1,...,x_n) \geq 
 ( \delta^Q(K) -\eta)^{n^2}\}.
 \end{equation}
There exists $n^*=n^*(\eta)$ such that for all $n>n^*$, 
$$Prob_n(K^{n}\setminus A_{n,\eta})\leq \Big(1-\frac{\eta}{2 \delta^Q(K)}\Big)^{n^2}\mu(K)^{n}.$$
\end{corollary}	
	
	We get the induced product probability measure ${\bf P}$ on the space of arrays on $K$, 
	$$\chi:=\{X=\{{\bf Z_{n}}:=(x_{n1},...,x_{nn})\in K^{n}\}_{n\geq 1}\},$$ 
	namely,
	$$(\chi,{\bf P}):=\prod_{n=1}^{\infty}(K^{n},Prob_n).$$ 
From standard arguments using the Borel-Cantelli lemma, we obtain: 

\begin{corollary}\label{416} With $K,Q$ and $\mu$ as above, for ${\bf P}$-a.e. array $X=\{{\bf Z_{n}}\}  \in \chi$, 
$$\mu_n:=\sum_{j=1}^n  \delta_{x_{nj}} \to \mu_{K,Q} \ \hbox{weak-* as } \ n \to\infty.$$
\end{corollary}

Define the probability measures ({\it one-point correlation functions}) 
$$d\tau_n(x):=\frac {1}{Z_n} R_1^{(n)}(x)d\mu(x)$$ where 
$$R_1^{(n)}(x):=\int_{K^{n-1}} |VDM_n^Q(x,x_2,...,x_n)|^2 d\mu(x_2) \cdots d\mu(x_n).$$
Using Corollary \ref{416}, we get the following deterministic result.

\begin{corollary}\label{417} With $K,Q$ and $\mu$ as above, 
$$\tau_n \to \mu_{K,Q} \ \hbox{as } \ n \to\infty.$$
\end{corollary}

\begin{proof} For $f\in C(K)$, we show $\int_K f d\tau_n\to  \int_Kfd\mu_{K,Q}$. Writing ${\bf Z_{n}}=\{x_{nj}\}$ and $ \mu_n=\frac{1}{n}\sum_{j=1}^{n} \delta_{x_{nj}}$, given $\epsilon >0$, let 
$$E_n:=\{X=\{{\bf Z_{n}}\}\in \chi: |\int_K fd\mu_n-\int_K fd\mu_{K,Q}|\geq \epsilon\}$$
and $F_n:=\{{\bf Z_n}\in K^{n}: |\int_K fd\mu_n-\int_K fd\mu_{K,Q}|\geq \epsilon$\}. From Corollary \ref{416},
$$Prob_n(F_n)={\bf P}(E_n)\leq {\bf P}(\cup_{m\geq n}E_m)\downarrow 0.$$
Hence $Prob_n(F_n)<\epsilon$ for $n$ large. Splitting $K^n$ into $F_n$ and $K^{n}\setminus F_n$ we obtain
$$|\int_{K^{n}}\bigl(\int_K fd\mu_n-\int_K fd\mu_{K,Q}\bigr)dProb_n({\bf Z_n})|< \epsilon + 2||f||_K \cdot \epsilon.$$
Thus
$$\int_{K^{n+1}}\int_K fd\mu_ndProb_n({\bf Z_n})\to \int_Kfd\mu_{K,Q}.$$
Using $\mu_n=\frac{1}{n}\sum_{j=1}^{n} \delta_{x_{nj}}$,
$$\int_{K^{n}}\int_K fd\mu_n dProb_n({\bf Z_n})=$$
$$\frac{1}{Z_n}\int_{K^{n}} \bigl(\frac{1}{n}\sum_{j=1}^{n} f(x_{nj})\bigr)|VDM_n^Q(x_{n1},...,x_{nn})|\prod_{j=1}^n d\mu(x_{nj})$$
$$= \frac{1}{Z_n}\int_{K}f(x_{n1}) \bigl(\int_{K^{n-1}}|VDM_n^Q(x_{n1},...,x_{nn})|\prod_{j=2}^n d\mu(x_{nj})\bigr)d\mu(x_{n1})$$
$$= \frac{1}{Z_n}\int_{K}f(x_{n1}) R_1^{(n)}(x_{n1}) d\mu(x_{n1}) =\int_K f d\tau_n.$$

\end{proof}
\section{Large deviation principle} We will need an approximation lemma to prove our large deviation result. 

\begin{lemma}\label{lemma-scal}
Let $K\subset \R^d$ be compact and of positive Riesz capacity and let $\tau\in\MM(K)$ with 
$I(\tau)<\infty$. There exist an increasing sequence of compact sets $K_m$ in $K$, a sequence of functions $\{Q_m\}\subset C(K)$, and a sequence of measures $\tau_m\in \MM(K_m)$ satisfying
\begin{enumerate}%[noitemsep,nolistsep]
\item the measures $\tau_m$ tend weakly to $\tau$, as $m\to\infty$; 
\item the energies $I(\tau_m)$ tend to $I(\tau)$ as $m\to\infty$;
\item  the measures $\tau_m$ are equal to the weighted equilibrium measures $\tau_{K,Q_m}$.
\end{enumerate}
\end{lemma}

\begin{proof} 
By Lusin's continuity theorem applied in $K$, for every integer $m\geq 1$, there exists a compact subset $K_m$ of $K$ such that $\tau(K\setminus K_m)\leq 1/m$ and $U^{\tau}|_{K_m}$ is continuous on $K_m$. We may assume that $K_m$ is increasing as $m$ tends to infinity. Then the measures $\tilde\tau_m:=\tau_{|K_m}$ are increasing and tend weakly to $\tau$. We have
$$\chi_m(x,y)\frac{1}{|x-y|^{\alpha}}\uparrow \frac{1}{|x-y|^{\alpha}} \ \hbox{as} \ m\to\infty$$
$(\tau\times\tau)$-almost everywhere on $K\times K$ where $\chi_m(x,y)$ is the characteristic function of $K_m\times K_m$ and we agree that the left-hand sides vanish when $x=y\notin K_{m}$. Hence, by monotone convergence we have
$$I(\tilde\tau_m)\to I(\tau) \quad\text{as }m\to\infty.$$

Next, define $\tau_m:= \tilde \tau_m/\tau(K_m)$. Clearly we have $I(\tau_m)\to I(\tau)$ since $I(\tilde\tau_m)\to I(\tau)$ and $\tau(K_m)\uparrow 1$. Define, for $x\in K$,
$$Q_m(x):= -U^{\tau_m}(x)=\frac{-1}{\tau(K_m)}\cdot U^{\tilde \tau_m}(x).$$ 
We first show $Q_{m}$ is continuous on $K_{m}$. Since $U^{\tau_m}$ is lower semicontinuous, it suffices to show it is upper semicontinuous. This follows since $U^{\tau-\tilde \tau_{m}}=U^{\tau}-U^{\tilde \tau_{m}}$ is lower semicontinuous (indeed, continuous) and $U^{\tau}$ is continuous on $K_{m}$. By the continuity property of Riesz potentials (\cite{La}, Theorem 1.7, p. 69, valid for all $0<\alpha <d$), we have $Q_m=-U^{\tau_m}$ is continuous on $\R^d$ (and in particular on $K$). Item 3. follows from Remark \ref{needed}. 
\end{proof}
We have all of the ingredients needed to follow the arguments of section 6 of \cite{LPLD} to prove the analogue of Theorem 6.6 there and hence a large deviation principle (Definition \ref{equivform} and Theorem \ref{ldp} below) which quantifies the statement of ${\bf P}$-a.e. convergence for arrays $X=\{{\bf Z_{n}}\}$ where ${\bf Z_{n}}=\{x_{nj}\}$ of $\frac{1}{n}\sum_{j=1}^n\delta_{x_{nj}}$ to $\mu_{K,Q}$. Given $G\subset {\mathcal M}(K)$, for each $n=1,2,...$ we let 
%$\tilde G_k$ be the subset of $K^{k}$ of elements
\begin{equation}\label{nbhddef}
\tilde G_n:=\{{\bf a} =(a_{1},...,a_{n})\in K^{n},~
\frac{1}{n}\sum_{j=1}^{{n}} \delta_{a_{j}}\in G\},
\end{equation}
and set
\begin{equation}\label{jkqmu}
J^Q_n(G):=\Big[\int_{\tilde G_{n}}|VDM^Q_n({\bf a})|d\nu ({\bf a})\Big]
^{ 1/n^2}.
\end{equation}
\begin{definition} \label{jwmuq} For $\mu \in \mathcal M(K)$ we define
$$\overline J^Q(\mu):=\inf_{G \ni \mu} \overline J^Q(G) \ \hbox{where} \ \overline J^Q(G):=\limsup_{n\to \infty} J^Q_n(G);$$
$$\underline J^Q(\mu):=\inf_{G \ni \mu} \underline J^Q(G) \ \hbox{where} \ \underline J^Q(G):=\liminf_{n\to \infty} J^Q_n(G)$$
where the infimum is taken over all open neighborhoods $G\subset \mathcal M(K)$ of $\mu$. If $Q=0$ we simply write $\overline J(\mu), \underline  J(\mu)$.
\end{definition}

We will only consider weights $Q\in C(K)$; thus the analogue of Lemma 6.3 in \cite{LPLD} is simpler:

\begin{lemma}\label{lem-J-JQ}
The following properties hold (and with the $\underline J,\underline J^Q$ 
functionals as well): 
\begin{enumerate}
\item $\overline J^Q(\mu)\leq e^{-I^Q(\mu)}$ for $Q\in C(K)$;
\item $ 
\overline J^{Q}(\mu)=\overline J(\mu)\cdot e^{-2\int_K Qd\mu}$  for $Q\in C(K)$. 
\end{enumerate}
\end{lemma}

Following the steps in section 6 of \cite{LPLD} with Corollary 5.3 there replaced by our approximation result, Lemma \ref{lemma-scal}, we obtain equality of the $ \overline J^{Q}$ and $ \underline J^{Q}$ functionals for any admissible weight $Q$ provided $\nu$ is a strong Bernstein-Markov measure for $K$ (see Theorem 6.6 in \cite{LPLD}). Note because of 2. in Lemma \ref{lem-J-JQ} we do not need the general Lemma 5.2 of 
\cite{LPLD}.
 \begin{theorem} \label{rel-J-E} 
 Let $K\subset \R^d$ be a compact set of positive Riesz capacity. Let $\nu\in {\mathcal M}(K)$ be a strong Bernstein-Markov measure for $K$ (e.g., if $K$ and $\nu$ satisfy the hypotheses of Theorem \ref{bmq}). \\
 (i) For any $\mu\in \mathcal M(K)$, 
$$
 \log \overline J(\mu)= \log\underline  J(\mu)=-I(\mu).$$
(ii) Let $Q\in C(K)$. Then for any $\mu\in \mathcal M(K)$, 
$$
 \overline J^{Q}(\mu)=\overline J(\mu)\cdot e^{-2\int_K Qd\mu}
$$
(and with the $\underline J,\underline J^Q$ 
functionals as well) so that,
 \begin{equation}\label{minwtd}\log \overline J^Q(\mu)= \log \underline J^Q(\mu)=-I^{Q}(\mu).
\end{equation}
\end{theorem}

\noindent Thus we simply write $J,J^Q$ without an underline or overline.

Define 
$j_n:  K^{n} \to \mathcal M(K)$ via 
\begin{equation}\label{jk} j_n(x_1,...,x_{n})%=\kappa_k(x):
=\frac{1}{n}\sum_{j=1}^{n} \delta_{x_j}.
\end{equation}
The push-forward
$\sigma_n:=(j_n)_*(Prob_n) $ is a probability measure on $\mathcal M(K)$: for a Borel set $G\subset \mathcal M(K)$,
\begin{equation}\label{sigmak}
\sigma_n(G)=\frac{1}{Z_n} \int_{\tilde G_{n}} |VDM_n^Q(x_1,...,x_{n})| d\nu(x_1) \cdots d\nu(x_{n}).
\end{equation}

\begin{definition} \label{equivform}
The sequence $\{\sigma_n\}$ of probability measures on $\mathcal M(K)$ satisfies a {\bf large deviation principle} (LDP) with good rate function $\mathcal I$ and speed $\{s_n\}$ with $s_n\to \infty$ if for all 
measurable sets $\Gamma\subset \mathcal M(K)$, 
\begin{equation}\label{lowb}-\inf_{\mu \in \Gamma^0}\mathcal I(\mu)\leq \liminf_{n\to \infty} \frac{1}{s_n} \log \sigma_n(\Gamma) \ \hbox{and}\end{equation}
\begin{equation}\label{highb} \limsup_{n\to \infty} \frac{1}{s_n} \log \sigma_n(\Gamma)\leq -\inf_{\mu \in \bar \Gamma}\mathcal I(\mu).\end{equation}
\end{definition}

\noindent On $\mathcal M(K)$, to prove a LDP it suffices to work with a base for the weak topology. The following is a special case of a basic general existence result, Theorem 4.1.11 in \cite{DZ}.

\begin{proposition} \label{dzprop1} Let $\{\sigma_{\epsilon}\}$ be a family of probability measures on $\mathcal M(K)$. Let $\mathcal B$ be a base for the topology of $\mathcal M(K)$. For $\mu\in \mathcal M(K)$ let
$$\mathcal I(\mu):=-\inf_{\{G \in \mathcal B: \mu \in G\}}\bigl(\liminf_{\epsilon \to 0} \epsilon \log \sigma_{\epsilon}(G)\bigr).$$
Suppose for all $\mu\in \mathcal M(K)$,
$$\mathcal I(\mu):=-\inf_{\{G \in \mathcal B: \mu \in G\}}\bigl(\limsup_{\epsilon \to 0} \epsilon \log \sigma_{\epsilon}(G)\bigr).$$
Then $\{\sigma_{\epsilon}\}$ satisfies a LDP with rate function $\mathcal I(\mu)$ and speed $1/\epsilon$. 
\end{proposition}

Following section 7 of \cite{LPLD}, Theorem \ref{rel-J-E} and Proposition \ref{dzprop1} immediately yield a large deviation principle:

\begin{theorem} \label{ldp} Assume $\nu$ is a strong Bernstein-Markov measure for $K$ and $Q\in C(K)$. The sequence $\{\sigma_n=(j_n)_*(Prob_n)\}$ of probability measures on $\mathcal M(K)$ satisfies a large deviation principle with speed $n^{2}$ and good rate function $\mathcal I:=\mathcal I_{K,Q}$ where, for $\mu \in \mathcal M(K)$,
\begin{equation*}%\label{ratefcnlform}
\mathcal I(\mu):=\log J^Q(\mu_{K,Q})-\log J^Q(\mu)=I^Q(\mu)-I^Q(\mu_{K,Q}).
\end{equation*}
%see (\ref{w=j=i}).
 \end{theorem}
 
 %\noindent Intuitively, this says the following. Given any $\mu\in \mathcal M(K)$ with $\mu \not = \mu_{K,Q}$, we know the probability that a random array $X=\{x_j^{(k)}\}_{j=1,...,n; \ n=2,3,...}\in \chi$ has the property that 
% $\frac{1}{n}\sum_{j=1}^n\delta_{x_j^{(k)}}\to \mu$ is zero; the ``rate'' at which the probabilty that this sequence lies in small neighborhoods of $\mu$ tends to zero as $n\to \infty$ like $\exp{\bigl[-n^{2}\cdot \mathcal I(\mu)\bigr]}$. 

\section{Open problems} We conclude with some questions which are motivated by logarithmic weighted potential theory as in \cite{ST}.
\begin{enumerate}
\item Where does the supremum norm of a function $f_n^Q\in  \mathcal P_n^Q$ (see (\ref{pnq})) live? More precisely, can one control the growth of $f_n^Q$ from its size on $S_w=$supp$(\mu_{K,Q})$ and/or on $S_w^*:=\{x\in K: U^{\mu_{K,Q}}(x)+Q(x)  \leq F_w\}$? 
\item Let $K$ be a compact set of positive Riesz capacity. Does there exist a positive measure $\mu$ on $K$ of finite total mass satisfying (\ref{BMQ}) for $Q=0$? Can such a $\mu$ be discrete? See \cite{DRNA} for related results.
\item Generalize the results in this paper to the case where $K\subset \R^d$ is allowed to be a closed, unbounded set for appropriate weights $Q$ (\cite{CGZ} considers $K=\R^d$).

\end{enumerate}

{\obeylines
\texttt{T. Bloom, bloom@math.toronto.edu
University of Toronto, Toronto, Ontario M5S 2E4 Canada
\medskip
N. Levenberg, nlevenbe@indiana.edu
Indiana University, Bloomington, IN 47405 USA
\medskip
F. Wielonsky, wielonsky@cmi.univ-mrs.fr
Universit\'e Aix-Marseille, CMI 39 Rue Joliot Curie
F-13453 Marseille Cedex 20, FRANCE }
}

\end{document}